\documentclass[a4paper,12pt]{article}
\usepackage{amsmath,amsfonts,amssymb,color}
\usepackage[english]{babel}
\usepackage{amsmath,amsfonts,amssymb,amscd,mathrsfs}
\usepackage{amsfonts}
\usepackage{amsthm}
\usepackage[all]{xy}

\newtheorem{teo}{Theorem}[section]
\newtheorem{prop}[teo]{Proposition}
\newtheorem{lem}[teo]{Lemma}
\newtheorem{coro}[teo]{Corollary}

\newtheorem{rem}[teo]{Remark}

\begin{document} 
\title{\vspace*{0cm}The restricted Lagrangian Grassmannian in infinite dimension}

\date{}
\author{Manuel L\'opez Galv\'an\footnote{Supported by Instituto Argentino de Matem\'atica (CONICET-PIP 2010-0757), Universidad Nacional de General Sarmiento and ANPCyT (PICT 2010-2478).}}

\maketitle
\setlength{\parindent}{0cm} 

\begin{abstract} In this paper we study the action of the symplectic operators which are a perturbation of the identity by a Hilbert-Schmidt operator in the Lagrangian Grassmannian manifold.
\end{abstract}

\section{Introduction}
In finite dimension, the Lagrangian Grassmannian $\Lambda(n)$ of the Hilbert space $\mathcal{H}=\mathbb{R}^{n}\times\mathbb{R}^{n}$ with the canonical complex structure $J(x,y)=(-y,x)$ was introduced by V.I. Arnold in 1967 \cite{Arnold}. These notions have been generalized to infinite dimensional Hilbert spaces (see \cite{Furutani}) and have found several applications to Algebraic Topology, Differential Geometry and Physics.

In classical finite dimensional Riemannian theory it is well-known the fact that given two points there is a minimal geodesic curve that joins them and this is equivalent to the completeness of the metric space with the geodesic distance; this is the Hopf-Rinow theorem. In the infinite dimensional case this is no longer true. In \cite{McAlpin} and \cite{Atkin}, McAlpin and Atkin showed in two examples how this theorem can fail. 

In \cite{Andruchow1} E. Andruchow and G. Larotonda introduced a linear connection in the Lagrangian Grassmannian and focused on the geodesic structure of this manifold. There they proved that any two Lagrangian subspaces can be joined by a minimal geodesic.   

In this paper we study a restricted version of the Lagrangian Grassmannian given by the action of the restricted symplectic group. We will focus on the geometric study and we will discuss which metric can be defined in each tangent space and which geometric properties it verifies.

\section{Background and definitions}
In this paper we will follow the notation and definitions of \cite{Manuel}, so first we recall some of this facts. Let $\mathcal{H}$ be an infinite dimensional real Hilbert space and let $\mathcal{B}(\mathcal{H})$ be the space of bounded operators. Denote by $\mathcal{B}_2(\mathcal{H})$ the Hilbert-Schmidt class $$\mathcal{B}_2(\mathcal{H})=\left\{a \in \mathcal{B}(\mathcal{H}): Tr(a^*a)< \infty\right\}$$ where $Tr$ is the usual trace in $\mathcal{B}(\mathcal{H})$. This space is a Hilbert space with the inner product $$<a,b>=Tr(b^*a).$$ The norm induced by this inner product is called the 2-norm and denoted by $$\Vert a \Vert_2=Tr(a^*a)^{1/2},$$ the usual operator norm will be denoted by $\Vert \ \Vert$. 

If $\mathcal{A}\subset\mathcal{B}(\mathcal{H})$ is any subset of operators we use the subscript $s$ (resp $as$) to denote the subset of symmetric (resp. anti-symmetric) operators of it, i.e. $\mathcal{A}_s=\left\{ x \in \mathcal{A}: x^*=x\right\}$ and $\mathcal{A}_{as}=\left\{ x \in \mathcal{A}: x^*=-x\right\}$. 

We fix a complex structure; that is a linear isometry $J \in \mathcal{B}(\mathcal{H})$ such that, $$J^2=-1  \ \mbox{and} \ J^*=-J.$$
The symplectic form $w$ is given by $w(\xi,\eta)=\left\langle J\xi,\eta \right\rangle$. We denote by $GL(\mathcal{H})$ the group of invertible operators and by ${\rm Sp}(\mathcal{H})$ the subgroup of invertible operators which preserve the symplectic form, that is $g \in {\rm Sp}(\mathcal{H})$ if $w(g\xi,g\eta)=w(\xi,\eta)$. Algebraically $${\rm Sp}(\mathcal{H})=\left\{ g \in GL(\mathcal{H}): g^*Jg=J\right\}.$$ This group is a Banach-Lie group (see \cite{Manuel}) and its Banach-Lie algebra is given by $$\mathfrak{sp}(\mathcal{H})=\left\{x \in \mathcal{B}(\mathcal{H}): xJ=-Jx^*\right\}.$$ Denote by $\mathcal{H}_J$ the Hilbert space $\mathcal{H}$ with the action of the complex field $\mathbb{C}$ given by $J$, that is; if $\lambda=\lambda_1+i\lambda_2 \in \mathbb{C}$ and $\xi \in \mathcal{H}$ we can define the action as $\lambda\xi:=\lambda_1\xi+\lambda_2J\xi$ and the complex inner product as $<\xi,\eta>_{\mathbb{C}}=<\xi,\eta>-iw(\xi,\eta)$. 

Let $\mathcal{B}(\mathcal{H}_J)$ be the space of bounded complex linear operators in $\mathcal{H}_J$. A straightforward computation shows that $\mathcal{B}(\mathcal{H}_J)$ consists of the elements of $\mathcal{B}(\mathcal{H})$ which commute with $J$.  

Following the notation of \cite{Manuel}, we consider the restricted subgroup of ${\rm Sp}(\mathcal{H})$ $${\rm Sp}_2(\mathcal{H})=\left\{ g \in {\rm Sp}(\mathcal{H}): g-1 \in \mathcal{B}_2(\mathcal{H})\right\}.$$
There it was proved that this group has a differentiable structure modelled on $\mathcal{B}_2(\mathcal{H})$. Some of the this facts are have been well-known for general Schatten ideals, more precisely the Banach-Lie group structure was noted in the book \cite{Harpe}.
The Lie algebra of ${\rm Sp}_2(\mathcal{H})$ is  $$\mathfrak{sp}_2(\mathcal{H})=\left\{x \in \mathcal{B}_2(\mathcal{H}): xJ=-Jx^*\right\}.$$ 

The Lagrangian Grassmannian $\Lambda(\mathcal{H})$ is the set of closed linear subspaces $L\subset \mathcal{H}$ such that $J(L)=L^{\perp}$. Clearly ${\rm Sp}(\mathcal{H})$ acts on $\Lambda(\mathcal{H})$ by means of $g.L=g(L)$. Since the action of the unitary group $U(\mathcal{H}_J)$ is transitive, it is clear that the action of ${\rm Sp}(\mathcal{H})$ is also transitive, so we can think of $\Lambda(\mathcal{H})$ as an orbit for a fixed $L_0\in \Lambda(\mathcal{H})$, i.e $$\Lambda(\mathcal{H})=\lbrace g(L_0) : g \in {\rm Sp}(\mathcal{H}) \rbrace.$$
We denote by $P_L \in \mathcal{B}(\mathcal{H})$ the orthogonal projection onto $L$. It is customary to parametrize closed subspaces via orthogonal projections, $L \leftrightarrow P_L$, in order to carry on geometric or analytic computations. We shall also consider here an alternative description of the Lagrangian subspaces using projections and symmetries. That is, $L$ is a Lagrangian subspace if and only if $P_LJ+JP_L=J$, see \cite{Furutani} for a proof. Another description of this equation using symmetries is $\epsilon_LJ=-J\epsilon_L$, where $\epsilon_L=2P_L-1$ is the symmetric orthogonal transformation which acts as the identity in $L$ and minus the identity in $L^{\perp}$. 

The isotropy subgroup at $L$ is $${{\rm Sp}(\mathcal{H})}_{L}=\lbrace g \in  {\rm Sp}(\mathcal{H}): g(L)=L\rbrace.$$ It is obvious that this subgroup is  a closed subgroup of ${{\rm Sp}(\mathcal{H})}$. In the infinite dimensional setting, this does not guarantee a nice submanifold structure; in Proposition \ref{sublie} we will prove that ${{\rm Sp}(\mathcal{H})}_{L}$ is a Banach-Lie subgroup of ${\rm Sp}(\mathcal{H})$.

We can restrict the natural action of the symplectic group in $\Lambda(\mathcal{H})$ to the restricted symplectic group and it will also be smooth. As before, we can consider the isotropy group at $L$ $${{\rm Sp}_2(\mathcal{H})}_{L}=\lbrace g \in {\rm Sp}_2(\mathcal{H}) : g(L)=L\rbrace.$$ We will also prove in Proposition \ref{sublie} that this subgroup is a Banach-Lie subgroup of ${{\rm Sp}_2(\mathcal{H})}$, with the topology induced by the metric $\Vert g_1-g_2 \Vert_2$ .

If $T$ is any operator we denote by ${Gr}_T$ its graph, i.e. the subset ${Gr}_T=\lbrace v+Tv : v \in \mbox{Dom}(T)\rbrace\subset \mathcal{H}\oplus\mathcal{H}$. Fix a Lagrangian subspace $L_0\subset \mathcal{H}$, we consider the subset of $\Lambda(\mathcal{H})$ $$\mathcal{O}_{L_0}=\lbrace g(L_0) : g \in {\rm Sp}_2(\mathcal{H})\rbrace\subseteq \Lambda(\mathcal{H}).$$ In Section 3 we will see that this set is strictly contained in $\Lambda(\mathcal{H})$ and thus the action of ${\rm Sp}_2(\mathcal{H})$ in the Lagrangian Grassmannian is not transitive. 
The purpose of this paper is the geometric study of this orbit; its manifold structure and relevant metrics.

\section{Manifold structure of $\mathcal{O}_{L_0}$}
We start proving that the subset $\mathcal{O}_{L_0}$ is strictly contained in $\Lambda(\mathcal{H})$, to do it we need the following lemma.

\begin{lem} If we identify the closed subspace $g(L_0)$ with its orthogonal projection $P_{g(L_0)}$ then it belongs to the affine space $P_{L_0}+\mathcal{B}_2(\mathcal{H})$. 
\end{lem}
\begin{proof} To prove it, we use the formula of the orthogonal projector over the range of an operator $Q$ given by 
\begin{equation} \label{proyort}
P_{R(Q)}=QQ^{*}{(1-(Q-Q^{*})^{2})}^{1/2}.
\end{equation}
This formula can be obtained using a block matrix representation. If we denote by $Q$ the idempotent associated of $g(L_0)$, i.e. $Q:=gP_{L_0}g^{-1}$ and if we suppose that $g=1+k$ and $g^{-1}=1+k'$ where $k,k'\in \mathcal{B}_2(\mathcal{H})$ we have $$QQ^{*}=(1+k)P_{L_0}(1+k')(1+k'^{*})P_{L_0}(1+k^{*})$$
\begin{align*}
&=\underbrace{(P_{L_0}+P_{L_0}k'+kP_{L_0}+kP_{L_0}k')}_Q \underbrace{(P_{L_0}+P_{L_0}k^{*}+k'^{*}P_{L_0}+k'^{*}P_{L_0}k^{*})}_{Q^{*}} \\
&=P_{L_0}+ \underbrace{P_{L_0}k^{*}+P_{L_0}k'^{*}P_{L_0}+.....}_{\in \ \mathcal{B}_2(\mathcal{H})}=P_{L_0}+T \  \in P_{L_0}+\mathcal{B}_2(\mathcal{H}). \\
\end{align*}     
It is clear that $Q-Q^{*} \in \mathcal{B}_2(\mathcal{H})$, then $(Q-Q^{*})^{2} \in \mathcal{B}_1(\mathcal{H})$. From the spectral theorem we have,  
$$1-(Q-Q^{*})^{2}=1+\sum_i \lambda_iP_i=P_0+\sum_i(\lambda_i+1)P_i$$ where $(\lambda_i) \in \ell^{1}$ and $P_0$ is the projection to the kernel. Taking square root, we have $${(1-(Q-Q^{*})^{2})}^{1/2}=P_0+\sum_i(\lambda_i+1)^{1/2}P_i$$ \begin{align*}
&=P_0+\sum_i[(\lambda_i+1)^{1/2}-1]P_i +\sum_i 1P_i \\
&=1+\sum_i[(\lambda_i+1)^{1/2}-1]P_i=1+T'  \  \in \ 1+ \mathcal{B}_2(\mathcal{H}) \\
\end{align*}
where $((\lambda_i+1)^{1/2}-1) \in \ell^{2}$, because $(\lambda_i) \in \ell^{1}$ and $\lim_{x\rightarrow 0}\dfrac{((x+1)^{1/2}-1)^{2}}{x}=0$. Then by the formula (\ref{proyort}) we have $$P_{g(L_0)}=(P_{L_0}+T)(1+T') \in P_{L_0}+\mathcal{B}_2(\mathcal{H}).$$
\end{proof} 

\begin{coro} The inclusion $\mathcal{O}_{L_0} \subset \Lambda(\mathcal{H})$ is strict.
\end{coro}
\begin{proof} We will see that in the generic example. Let $\mathcal{K}$ be a Hilbert space and $\mathcal{H}=\mathcal{K}\times \mathcal{K}$ with the usual inner product. Let $J : \mathcal{H} \rightarrow \mathcal{H}$ given by $J(\xi,\eta)=(-\eta,\xi)$. We can take $L_0=\lbrace 0 \rbrace \times \mathcal{K}$ and ${Gr}_I$ the graph of the identity map of $\mathcal{K}$; this subspaces are Lagrangian with respect to the form $J$. If we write the orthogonal projector over this subspaces in terms of the decomposition, $L_0\oplus L_0^{\perp}$ we have $$ P_{L_0}=\begin{pmatrix}
  1 & 0 \\
  0 & 0 \\
\end{pmatrix} 
,\  P_{{Gr}_I}=\dfrac{1}{2}\begin{pmatrix}
  1 & 1 \\
  1 & 1 \\
\end{pmatrix}.$$ 
Suppose that  $\Lambda(\mathcal{H})=\mathcal{O}_{L_0}$, then ${Gr}_I$ belongs to $\mathcal{O}_{L_0}$ and by the above Lemma its orthogonal projector can be written as an element of $P_{L_0}+\mathcal{B}_2(\mathcal{H})$, that is $P_{Gr_I}-P_{L_0}=T \in \mathcal{B}_2(\mathcal{H})$ and if we write it in terms of matrix blocks we have;
$$\dfrac{1}{2}\begin{pmatrix}
  1 & 1 \\
  1 & 1 \\
\end{pmatrix} -\begin{pmatrix}
  1 & 0 \\
  0 & 0 \\
\end{pmatrix} =\begin{pmatrix}
  P_{L_0}TP_{L_0} &  (1-P_{L_0})TP_{L_0}\\
  P_{L_0}T(1-P_{L_0}) & (1-P_{L_0})T(1-P_{L_0}) \\  
\end{pmatrix} \in \mathcal{B}_2(\mathcal{H})$$ and this is a contradiction because $-1/2=P_{L_0}TP_{L_0} \in \mathcal{B}_2(L_0)$ and $-1/2 \notin \mathcal{B}_2(L_0)$.
\end{proof}
To build a manifold structure over $\mathcal{O}_{L_0}$, we will considerate the charts of $\Lambda(\mathcal{H})$ given by the parametrization of Lagrangian subspaces as graphs of functions and we will adapt this charts to our set. This charts were used in \cite{Piccione1} to describe the manifold structure of $\Lambda(\mathcal{H})$; in the followings steps we recall this charts and we fix the notation.\\ Given $L \in \Lambda(\mathcal{H})$, we have the Lagrangian decomposition $\mathcal{H}=L\oplus L^{\perp}$ and we denote by $$\Omega({L}^{\perp})=\lbrace W \in \Lambda(\mathcal{H}) :  \mathcal{H}=W\oplus {L}^{\perp} \rbrace.$$ In \cite{Furutani} it was proved that these sets are open in $\Lambda(\mathcal{H})$.   
We consider the map  $\phi_{L} : \Omega({L}^{\perp}) \rightarrow \mathcal{B}(L)_s$ given by $$W={Gr}_T \longmapsto J\vert_{{L}^{\perp}}T$$ 
where $T: L \rightarrow {L}^{\perp}$ is the linear operator whose graph is $W$, more precisely $$T=\pi_1\vert_W \circ (\pi_0\vert_W)^{-1}$$ where $\pi_0$, $\pi_1$ are the orthogonal projections to $L$ and ${L}^{\perp}$.  
\begin{rem}\label{carsobre} The map  $\phi_{L}$ is onto.
\end{rem}
\begin{proof} Let $\psi \in \mathcal{B}(L)_s$, we consider the operator $T:= -J\vert_{L}\psi$ ($T$ maps $L$ into $L^{\perp}$) and $W:={Gr}_T$. Since $\psi$ is a symmetric operator, $W$ is a Lagrangian subspace and $\mathcal{H}={Gr}_T\oplus {L}^{\perp}$; for this $W \in \Omega({L}^{\perp})$ and it is a preimage of $\psi$. 
\end{proof}  

The maps $\lbrace\phi_{L}\rbrace_{L\in \Lambda(\mathcal{H})}$ constitute a smooth atlas for $\Lambda(\mathcal{H})$, so that $\Lambda(\mathcal{H})$ becomes a smooth Banach manifold (see \cite{Piccione2}). For every $W \in \Lambda(\mathcal{H})$ we can identify the tangent space $T_W\Lambda(\mathcal{H})$ with the Banach space $\mathcal{B}(W)_s$, this identification was used in \cite{Piccione1} and \cite{Piccione2}. For $W \in \Omega({L}^{\perp})$, the differential $d \phi_{L}$ of the chart  at $W$ is given by 
\begin{equation}\label{difdecartas}
d_W\phi_{L}(H)=\eta^{*}H\eta
\end{equation}  for  all $H \in \mathcal{B}(W)_s$, where $\eta : L \rightarrow W$ is the isomorphism given by the restriction to $L$ of the projection $W\oplus L^{\perp} \rightarrow W$. It is easy to see that the inverse $d_\psi\phi_L^{-1}$ of this map at a point $\psi =\phi_L (W)$ is given by  
$$\mathcal{B}(L)_s \stackrel{d_\psi\phi_L^{-1}} \longrightarrow \mathcal{B}(W)_s$$   
$$H \longmapsto {(\eta^{-1})}^*H\eta^{-1}.$$ 

Since the symplectic group acts smoothly we can consider for fixed $L \in \Lambda(\mathcal{H})$ the smooth map $\pi_{L} : {\rm Sp}(\mathcal{H}) \rightarrow \Lambda(\mathcal{H})$ given by $g \mapsto g(L)$. Its differential map at a point $g\in {\rm Sp}(\mathcal{H})$ is given by $$T_g{\rm Sp}(\mathcal{H})=\mathfrak{sp}(\mathcal{H})g \ni Xg\mapsto P_{g(L)}JX\vert_{g(L)} \in \mathcal{B}(g(L))_s,$$ see \cite{Piccione1} and \cite{Piccione2} for a proof.  
If $L\in \mathcal{O}_{L_0}$ we can restrict the map $\pi_{L}$ to the subgroup ${\rm Sp}_2(\mathcal{H})$ obtaining a surjective map onto $\mathcal{O}_{L_0}$,
$$\pi_{L}\vert_{{\rm Sp}_2(\mathcal{H})} : {\rm Sp}_2(\mathcal{H}) \rightarrow \mathcal{O}_{L_0}.$$

\begin{teo} \label{cartasOL_0}The set $\mathcal{O}_{L_0}$ is a submanifold of $\Lambda(\mathcal{H})$ and the natural map $i :\mathcal{O}_{L_0} \hookrightarrow \Lambda(\mathcal{H})$ is an immersion.
\end{teo}
\begin{proof} We will adapt the above local chart $\phi_{L}$ to our set. Let $L=g(L_0)\in \mathcal{O}_{L_0}$, first we see that $\phi_{L}(\Omega({L}^{\perp}) \cap \mathcal{O}_{L_0}) \subset \mathcal{B}_2(L)_s$. Indeed, if  $W$ belongs to $\Omega({L}^{\perp}) \cap \mathcal{O}_{L_0}$ then we can write $W ={Gr}_T=h(L_0)$ for some $h\in{\rm Sp}_2(\mathcal{H})$ and since $L_0=g^{-1}(L)$ we have that $W=hg^{-1}(L)$ and it is obvious that we can write now $W=\tilde{g}(L)$ with $\tilde{g} \in {\rm Sp}_2(\mathcal{H})$. If we write $\tilde{g}=1+k$ where $k\in \mathcal{B}_2(\mathcal{H})$ then the orthogonal projection $\pi_1$ restricted to $W$ can be written as 
$$\pi_1\vert_W(w)=\pi_1(\tilde{g}l)=\pi_1(l+kl)=\pi_1(k(l))=\pi_1(k(\tilde{g}^{-1}w))$$ 
where $W \ni w=\tilde{g}(l)$ and $l\in L$. Thus we have $$\pi_1\vert_W=\pi_1 \circ k\circ\tilde{g}^{-1}\vert_W \in \mathcal{B}_2(W,L^{\perp}).$$
Then it is clear that  $\phi_{L}(W)=J\vert_{{L}^{\perp}}T \in \mathcal{B}_2(L)_s.$ Now we have the restricted chart $$\phi_{L} \vert _{\Omega({L}^{\perp}) \cap\mathcal{O}_{L_0}} : \Omega({L}^{\perp}) \cap\mathcal{O}_{L_0} \longrightarrow \mathcal{B}_2(L)_s.$$
To conclude we will see that this restricted map is also onto. Let $\psi \in \mathcal{B}_2(L)_s$ and as we did in Remark \ref{carsobre} we consider the operator $T:=-J\vert_{L}\psi$, then the only fact to prove is that $${Gr}_T=\lbrace v+(-J\vert_{L}\psi)v : v \in L \rbrace \in \mathcal{O}_{L_0}.$$ 
To prove it we define $f:=1-J\vert_{L}\psi P_{L}\in 1+ \mathcal{B}_2(\mathcal{H})$; it is invertible with inverse given by $1+J\vert_{L}\psi P_{L}$ and it is clear that ${Gr}_T=f(L).$ Now we have to show that $f$ is symplectic. Indeed, let $\xi,\eta \in \mathcal{H}$
 $$w((1-J\vert_{L}\psi P_{L})\xi,(1-J\vert_{L}\psi P_{L})\eta)= $$ 
$$w(\xi,\eta)+w(\xi,-J\vert_{L}\psi P_{L}\eta)+w(-J\vert_{L}\psi P_{L}\xi,\eta)+\underbrace{w(J\vert_{L}\psi P_{L}\xi,J\vert_{L}\psi P_{L}\eta)}_{=0} $$
and since $J$ is an isometry we have, 

\begin{align*}
w(\xi,-J\vert_{L}\psi P_{L}\eta)+w(-J\vert_{L}\psi P_{L}\xi,\eta) &= \langle J\xi,-J\vert_{L}\psi P_{L}\eta\rangle + \langle J(-J\vert_{L}\psi P_{L})\xi,\eta\rangle \\
&=  -\langle \xi,\psi P_{L}\eta\rangle+\langle\psi P_{L}\xi,\eta\rangle.\\
\end{align*}

If $\xi=\xi_0+\xi_0^{\perp}$ and $\eta=\eta_0+\eta_0^{\perp}$ are the respective decompositions in $L\oplus L^{\perp}$, then by the symmetry of $\psi$ the above equality results in
$$-\langle \xi,\psi P_{L}\eta\rangle+\langle\psi P_{L}\xi,\eta\rangle=\langle \xi_0+\xi_0^{\perp},\psi\eta_0\rangle + \langle \psi\xi_0,\eta_0+\eta_0^{\perp}\rangle$$
$$=-\langle \xi_0,\psi\eta_0\rangle + \langle \psi\xi_0,\eta_0\rangle=0.$$Then $$w((1-J\vert_{L}\psi P_{L})\xi,(1-J\vert_{L}\psi P_{L})\eta)=w(\xi,\eta)$$ and $f \in {\rm Sp}_2(\mathcal{H}).$ Since $L=g(L_0)$ we have $${Gr}_T=f(L)=fg(L_0) \in \mathcal{O}_{L_0}.$$
\end{proof}
As in the case of the full Lagrangian Grassmannian, for every $L\in \mathcal{O}_{L_0}$ we can identify the tangent space $T_L\mathcal{O}_{L_0}$ with the Hilbert space $\mathcal{B}_2(L)_s$. 

Since the differential of the inclusion map is an inclusion map, it is clear that the differential of the adapted charts are the restriction of the differential of full charts given by equation (\ref{difdecartas}). So, if $W \in \Omega(L^{\perp})\cap \mathcal{O}_{L_0}$ then the differential of the adapted chart is given by $d_W\phi_{L}\vert_{\Omega({L}^{\perp}) \cap\mathcal{O}_{L_0}}(H)=\eta^{*}H\eta$ where $H \in \mathcal{B}_2(W)_s$ and its inverse is  $$\mathcal{B}_2(L)_s \stackrel{d_\psi\phi_L^{-1}\vert_{\Omega({L}^{\perp}) \cap\mathcal{O}_{L_0}}  } \longrightarrow \mathcal{B}_2(W)_s=T_W \mathcal{O}_{L_0}$$   
\begin{equation}\label{difcartas2}
H \longmapsto {(\eta^{-1})}^*H\eta^{-1}. 
\end{equation}

\begin{prop}\label{difpiL} The differential of the map $\pi_{L}\vert_{{\rm Sp}_2(\mathcal{H})}$ at a point $g \in {{\rm Sp}_2(\mathcal{H})}$ is the restriction of the differential map $d_g\pi_{L}$ at $T_g{\rm Sp}_2(\mathcal{H})$ i.e.   
$$d_g\pi_{L}\vert_{{\rm Sp}_2(\mathcal{H})}: T_g{\rm Sp}_2(\mathcal{H})=\mathfrak{sp}_2(\mathcal{H})g \ni Xg\mapsto P_{g(L)}JX\vert_{g(L)} \in \mathcal{B}_2(g(L))_s.$$
\end{prop}
\begin{proof} We have the following commutative diagram 
$$\xymatrix{
{\rm Sp}(\mathcal{H}) \ar[r]^{\pi_{L}} &\Lambda(\mathcal{H})\\
{\rm Sp}_2(\mathcal{H}) \ar@{_{(}->}[u]^{i_2} \ar[r]_{\pi_{L}\vert_ {{\rm Sp}_2(\mathcal{H})}} &\ \ar@{_{(}->}[u]^{i_1} \mathcal{O}_{L_0}} $$ 
If we derive at a point $g \in {\rm Sp}_2(\mathcal{H})$ the equation $\pi_{L}\circ i_2=i_1\circ \pi_{L}\vert_{{\rm Sp}_2(\mathcal{H})}$ and use that the differential of the inclusion maps $i_1$ and $i_2$ at $h(L_0)$ and at $h$ respectively are inclusions, we have $d_g\pi_{L}\vert_{{\rm Sp}_2(\mathcal{H})}(Xg)=d_g\pi_{L}(Xg)$ for every $X \in \mathfrak{sp}_2(\mathcal{H})$.     
\end{proof}

In the followings steps we will show the main result of this section, that is the Lie subgroup structure of the isotropy group. To do it we will use the above submanifold structure constructed over $\mathcal{O}_{L_0}$. If $M$ and $N$ are smooth Banach manifolds a smooth map $f:M \rightarrow N$ is a submersion if the tangent map $d_xf$ is onto and its kernel is a complemented subspace of $T_xM$ for all $x\in M$. This fact is equivalent to the existence of smooth local section (see \cite{Lang}). The next proposition is essential for the proof.    
\begin{prop} The map $\pi_{L_0}$ and its restriction $\pi_{L_0}\vert_{{\rm Sp}_2(\mathcal{H})}$ are smooth submersions when we consider in $\Lambda(\mathcal{H})$ (resp. in $\mathcal{O}_{L_0}$) the above manifold structure.
\end{prop}
\begin{proof} First we will prove that the map $\pi_{L_0}\vert_{{\rm Sp}_2(\mathcal{H})}$ has local cross sections on a neighborhood of $L_0$, the proof is adapted from \cite{Andruchow1}. Using the symmetry over $R(Q)$ we have \begin{equation} \label{e1}\epsilon_{R(Q)}=2P_{R(Q)}-1=2(P_{L_0}+\mathcal{B}_2)-1=\epsilon_{L_0}+\mathcal{B}_2.\end{equation} For $L\in \mathcal{O}_{L_0}$ close to $L_0$, we consider the element  $g_L=1/2(1+\epsilon_{L}\epsilon_{L_0})$; it is invertible (in fact, it can be show that it is invertible if  $\Vert\epsilon_{L}-\epsilon_{L_0}\Vert<2$) and it commutes with $J$, so it belongs to $GL(\mathcal{H}_J)$. From equation (\ref{e1}) we have $$\epsilon_{L}\epsilon_{L_0} \in (\epsilon_{L_0}+\mathcal{B}_2(\mathcal{H}))\epsilon_{L_0} \in 1+\mathcal{B}_2(\mathcal{H})$$ and then it is clear that $g_L \in 1+\mathcal{B}_2(\mathcal{H}_J).$ Thus $g_L$ is complex and invertible in a neighboord of $\epsilon_{L_0}$. Note that     
$$g_L\epsilon_{L_0}=1/2(\epsilon_{L_0}+\epsilon_{L})=\epsilon_{L}g_L$$ and also that $g^*g$ commutes with $\epsilon_{L_0}$. If $\left|x\right|=(x^*x)^{1/2}$ denote the absolute value and $g_L=u_L\left|g_L\right|$ is the polar decomposition, then $u_L=g_L({g_L}^{*}g_L)^{-1/2} \in U(\mathcal{H}_J)\subset {\rm Sp}(\mathcal{H})$. We define the local cross section for $L$ close to $L_0$ as  $$\sigma(L)=u_L.$$ Now we have to prove that $\pi_{L_0}\vert_{{\rm Sp}_2(\mathcal{H})}(\sigma (L))=L$. If we identify the subspace with the symmetry this is equivalent to prove that $\epsilon_{\pi_{L_0} \vert_{{\rm Sp}_2(\mathcal{H})}  (\sigma(L))}=\epsilon_{L}$. Indeed,  $$\epsilon_{\pi_{L_0}(u_L)}=u_L\epsilon_{L_0}u_L^{*}=g_L{(g_L^{*}g_L)}^{-1/2}\epsilon_{L_0}(g_L^{*}g_L)^{-1/2}g_L^{*}=g_L\epsilon_{L_0}g_L^{-1}=\epsilon_L.$$
Let us prove that it takes values in ${\rm Sp}_2(\mathcal{H})$. Since $\mathbb{C}1+\mathcal{B}_2(\mathcal{H}_J)$ is a *-Banach algebra and $g_L \in GL_2(\mathcal{H}_J)$ by the Riesz functional calculus we have that $u_L=g_L\left|g_L\right|^{-1} \in \mathbb{C}1+\mathcal{B}_2(\mathcal{H}_J)$. Thus $u_L=\beta 1+b$ with $b\in \mathcal{B}_2(\mathcal{H}_J)$. On the other hand, note that ${g_L}^{*}g_L$ is a positive operator which lies in the C*-algebra 
$\mathbb{C}1+\mathcal{K}(\mathcal{H}_J)$. Therefore its square root is of the form $r1+k$ with $r\geq 0$ and $k$ compact. Then $${g_L}^{*}g_L=(r1+k)^{2}=r^{2}+k'$$ and since ${g_L}^{*}g_L \in GL_2(\mathcal{H}_J)$ we have $$r^{2}1+k'=1+b'$$ with $b' \in \mathcal{B}_2(\mathcal{H}_J)$. Since  $\mathbb{C}1$ and $\mathcal{K}(\mathcal{H}_J)$ are linearly independent, it follows that $r=1$. Then it is clear that $u_L \in U_2(\mathcal{H}_J)\subset {\rm Sp}_2(\mathcal{H})$ and $\sigma$ is well defined. To conclude the proof we show that the local section $\sigma$ is smooth. If $L$ lies in a small neighborhood of $L_0$ we have $$L=\phi^{-1}_{L_0}(\psi)=Gr_{-J\vert_L \psi}=(1-J\vert_{L_0}\psi P_{L_0})(L_0)=g(L_0)  \ \in \  \Omega({L_0}^{\perp})\cap \mathcal{O}_{L_0}.$$ The idempotent of range $L$ is $$Q:=gP_{L_0}g^{-1}=(1-J\vert_{L_0}\psi P_{L_0}) P_{L_0}(1+J\vert_{L_0}\psi P_{L_0})=P_{L_0}-J\vert_{L_0}\psi P_{L_0}$$ and it is smooth as a function of $\psi$. Since the formula of the orthogonal projector (\ref{proyort}) is smooth, the local expression of $\sigma$ will be also smooth. Indeed, the symmetry in the chart will be $$\epsilon_L=2P_{R(gP_{L_0}g^{-1})}-1=2QQ^{*}{(1-(Q-Q^{*})^{2})}^{1/2}-1$$ and it is clearly smooth as a function of $\psi$, because $Q$ and the operations involved (product, involution, square root) are smooth. Then it is clear that the invertible element $g_L$ and its unitary part $u_L$ are smooth too. Finally the local expression $\sigma\circ\phi^{-1}_{L_0}$ is smooth as a function of $\psi$.  
Since the full Lagrangian Grassmannian can be expressed as an orbit for a fixed $L_0$, the proof of smoothness of the local section of $\pi_{L_0}$ is analogous to that of the restricted map $\pi_{L_0}\vert_{{\rm Sp}_2(\mathcal{H})}$.
\end{proof}

\begin{coro} If $L$ is any subspace in the full Lagrangian Grassmannian or in $\mathcal{O}_{L_0}$ then the map $\pi_L$ and its restriction $\pi_L\vert_{{\rm Sp}_2(\mathcal{H})}$ have local cross sections on a neighborhood of $L$.
\end{coro}
\begin{proof} The above map $\sigma$ can be translated using the action to any $L=g(L_0)$. That is, $$\sigma_L(h(L_0))=g\sigma(g^{-1}h(L_0))g^{-1}$$ where $h(L_0)$ lies on a neighborhood of $L$.
\end{proof}

\begin{teo} \label{sublie} The isotropy groups ${{\rm Sp}(\mathcal{H})}_{L}$ and ${{\rm Sp}_2(\mathcal{H})}_{L}$ of the symplectic group and of the restricted symplectic group are Lie subgroups of them with their respective topology. Their Lie algebras are $$\mathfrak{sp}(\mathcal{H})_{L}=\lbrace x \in \mathfrak{sp}(\mathcal{H}) : x(L)\subseteq L\rbrace$$  $$ \mathfrak{sp}_2(\mathcal{H})_{L}=\lbrace x\in \mathfrak{sp}_2(\mathcal{H}): x(L)\subseteq L\rbrace.$$
\end{teo}
\begin{proof} Since the maps $d_1\pi_{L}$ and $d_1\pi_{L}\vert_{{\rm Sp}_2(\mathcal{H})}$ are submersions then by the inverse function theorem, we have that the isotropy groups are Lie subgroups and their Lie algebras are $\ker d_1\pi_{L}$ and $\ker d_1\pi_{L}\vert_{{\rm Sp}_2(\mathcal{H})}$ respectively. A short computation show us that $\ker d_1\pi_{L}=\lbrace x \in \mathfrak{sp}(\mathcal{H}) : x(L)\subseteq L\rbrace$ and $\ker d_1\pi_{L}\vert_{{\rm Sp}_2(\mathcal{H})}=\lbrace x\in \mathfrak{sp}_2(\mathcal{H}): x(L)\subseteq L\rbrace.$ Indeed, if $P_{L}JX\vert_{L}=0$ then  $JX\vert_{L} \in L^{\perp}$ and thus $-X\vert_{L} \in J(L^{\perp})=L$.
\end{proof}
\begin{rem} The Lie algebra  $\mathfrak{sp}_2(\mathcal{H})_{L}$ consists of all operators $x \in \mathfrak{sp}_2(\mathcal{H})$ that are $L$ invariant, so we can give another characterization of this algebra using the orthogonal projection $P_L$. That is,   
\begin{equation} \label{lieiso} \mathfrak{sp}_2(\mathcal{H})_{L}=\lbrace x\in \mathfrak{sp}_2(\mathcal{H}): xP_L=P_LxP_L \rbrace.\end{equation} In block matrix form, this operators corresponding to the upper triangular elements of $\mathfrak{sp}_2(\mathcal{H})$.   
\end{rem}

\section{Metrics structures in $\mathcal{O}_{L_0}$}
\subsection{The ambient metric}
Given $v,w \in T_W \mathcal{O}_{L_0}=\mathcal{B}_2(W)_s$, we define the inner product $$ \langle v , w \rangle_W:=tr_W(w^{*}v)=\sum_{i=1}^\infty   \langle w^{*}v e_i , e_i \rangle$$ where $\lbrace e_i\rbrace$ is an orthonormal basis of the subspace $W$. It is not difficult to see that the above inner product can be expressed over the full Hilbert space $\mathcal{H}$ using the orthogonal projection $P_W$. That is, $tr_W(w^{*}v)=Tr(w^{*}vP_W).$ The ambient metric for $v \in T_W \mathcal{O}_{L_0}=\mathcal{B}_2(W)_s$ is $$\mathcal{A}(W,v):=tr_W(v^{*}v)^{1/2},$$ and it can be expressed by the orthogonal projection over the full Hilbert space. Indeed, if $\lbrace e_i\rbrace$ is an orthonormal basis for $\mathcal{H}$ then  
$$\Vert vP_W\Vert_2^2=\sum_{i} \langle vP_W e_i , vP_We_i \rangle=\sum_i \langle v^{*}vP_W e_i , P_We_i \rangle $$
\begin{equation}\label{full}
=\sum_i \langle v^{*}vP_W e_i , e_i \rangle=tr(v^{*}vP_W)=tr_W(v^{*}v)
\end{equation}
\begin{prop} The ambient metric is smooth in $\mathcal{O}_{L_0}$.
\end{prop}
\begin{proof} Let $L\in \mathcal{O}_{L_0}$ and consider a neighborhood  $U:=\Omega(L^{\perp})\cap \mathcal{O}_{L_0}$ of it. For any $W \in U$, we can write it in the local chart $W=\phi_L^{-1}\psi=Gr_{(-J\vert_L\psi)}$. Let  $\eta_W : L \rightarrow W$ be the restriction of the orthogonal projection  $W\oplus L^{\perp} \stackrel{\pi}\rightarrow W$, then its local expression is; $$\eta_W(v)=\pi(v)=\pi ((v-J\vert_L\psi(v))+J\vert_L\psi (v))$$  
$$=(1-J\vert_L\psi) (v)  \  \ \  \mbox{for all} \ v \in L,$$  
and then it can be expressed by the compression of the operator $1-J\vert_L\psi P_L$ into the subspace $L$ i.e. $\eta_W=(1-J\vert_L\psi P_L)\vert_L$.
If we write the local expression of the metric using the classical differential structure of the tangent bundle with the differential of the chart $\phi_L^{-1} $ given in the formula (\ref{difcartas2}), for every $v \in TU$ we have \begin{equation}\label{localmetrica}
 \mathcal{A}(W,v)=\Vert d_\psi\phi_L^{-1}(H)P_W\Vert_2=
\Vert (\eta_W^{-1})^{*} H \eta_W^{-1} P_W \Vert_2, 
\end{equation}
where $\psi \in \phi_L(U)$ and $H \in \mathcal{B}_2(L)_s$ is the preimage of $v$. Since the projector $P_W=P_{Gr_{(-J\vert_L\psi)}}$ is smooth and the local expression of $\eta_W$ is also smooth as a function of $\psi$ and by smoothness of the operations involved (inverse, involution, product, trace) the formula (\ref{localmetrica}) is smooth. 
\end{proof}
\subsection{The geodesic distance}
The length of a smooth curve measured with the ambient metric will be denoted by $$L_{\mathcal{A}}(\gamma)=\int_0^1 \mathcal{A}(\gamma(t),\dot{\gamma}(t)) dt.$$  
Given two Lagrangian subspaces $S$ and $T$ in $\mathcal{O}_{L_0}$, we denote by $d_{\mathcal{A}}$ the geodesic distance using the ambient metric, $$d_{\mathcal{A}}(S,T)=\inf \lbrace L_{\mathcal{A}}(\gamma) : \gamma \ \mbox{joins} \ S  \ \mbox{and} \  T  \ \mbox{in} \ \mathcal{O}_{L_0}\rbrace.$$ 
In this section we will show that the metric space $(\mathcal{O}_{L_0},d_{\mathcal{A}})$ is complete and moreover we will find the geodesic curves of the Riemannian connection given by the ambient metric $\mathcal{A}$.
 
If $(L_n)\subset \mathcal{O}_{L_0}$ is any sequence we will denote by $L_n \stackrel{\mathcal{O}_{L_0}}\rightarrow L$ the convergence to some subspace $L\in \mathcal{O}_{L_0}$ in the topology given by the smooth structure of $\mathcal{O}_{L_0}$ (Theorem \ref{cartasOL_0}). 
\begin{prop}\label{convda} Let $(L_n)\subset \mathcal{O}_{L_0}$ such that  $L_n \stackrel{\mathcal{O}_{L_0}}\rightarrow L$; then $L_n \stackrel{d_{\mathcal{A}}}\longrightarrow L.$
\end{prop}
\begin{proof} Since the map $\pi_L$ has local continuous sections, let $n_0$ such that $L_n \in U \subset\mathcal{O}_{L_0} \ \forall n \geq n_0$ ($U$ a neighboord of $L$) and such that $\sigma_L : U \rightarrow {\rm Sp}_2(\mathcal{H})$ is a section for $\pi_L$. By continuity we have  $\sigma_L(L_n) \stackrel{\Vert . \Vert_2}\longrightarrow \sigma_L(L)=1$ if $n\geq n_0$. Since $\sigma_L(L_n)$ is close to $1$, there is $z_n \in \mathfrak{sp}_2(\mathcal{H})$ such that $\sigma_L(L_n)=e^{z_n}$ and since $\Vert e^{z_n} -1\Vert_2=\Vert \sigma_L(L_n) -1 \Vert_2\rightarrow 0$ we also have $\Vert z_n \Vert_2 \rightarrow 0$. Let $\gamma_n(t)=e^{tz_n}(L)\subset \mathcal{O}_{L_0}$ be a curve that joins $L$ and $L_n$; using the equality (\ref{full}) its length is $L_{\mathcal{A}}(\gamma_n)=\int_0^{1} \mathcal{A}(\gamma_n(t),\dot{\gamma}_n(t)) dt=\int_0^1\Vert \dot{\gamma}_n(t)P_{\gamma_n(t)}\Vert_2 $. Since $\gamma_n(t)=\pi_L\circ e^{tz_n}$ using the chain rule and Proposition \ref{difpiL} we have 
$$ \dot{\gamma}_n(t)=d_{e^{tz_n}}\pi_L(z_ne^{tz_n})=P_{e^{tz_n}(L)}Jz_n\vert_{e^{tz_n}(L)},$$then taking norm and using the symmetric property of the Frobenius norm ($\Vert xyz\Vert_2\leq \Vert x\Vert\Vert y\Vert_2\Vert z\Vert$) we have $$ \Vert \dot{\gamma}_n(t)P_{\gamma_n(t)}\Vert_2 =\Vert P_{e^{tz_n}(L)}Jz_nP_{e^{tz_n}(L)}\Vert_2\leq \Vert z_n\Vert_2.$$ Then it is clear that $d_{\mathcal{A}}(L_n,L)\leq L_{\mathcal{A}}(\gamma_n) \rightarrow 0$.
\end{proof} 
Given a smooth curve $\alpha$ in ${\rm Sp}_2(\mathcal{H})$ we can measure its length with the left or right invariant metric, depending on which identification of tangent spaces we use in the group. In \cite{Manuel} it was used the left one, hence their use the left invariant metric. The length of a curve using this metric is $L_{\mathcal{L}}(\alpha)=\int_0^1\Vert \alpha^{-1}\dot{\alpha}\Vert_2$. In this paper we will use the right identification of the tangent spaces, so we have to introduce the right invariant metric. Although formally equivalent this choice will make some completeness easier. Then the length of $\alpha$ is, $L_{\mathcal{R}}(\alpha)=\int_0^1\Vert \dot{\alpha}\alpha^{-1}\Vert_2$.

\begin{prop}\label{dldr} If $d_{\mathcal{L}}$ and $d_{\mathcal{R}}$ denote the geodesic distance with the left and right invariant metrics respectively then, $$d_{\mathcal{L}}(x^{-1},y^{-1})=d_{\mathcal{R}}(x,y)  \   \   \  \forall  \  x,y \in  {\rm Sp}_2(\mathcal{H}).$$
\end{prop}
\begin{proof} Since the geodesic distances are left and right invariant respectively, the only fact left to prove is the equality $d_{\mathcal{L}}(x^{-1},1)=d_{\mathcal{R}}(x,1)$ for all $x \in {\rm Sp}_2(\mathcal{H})$. Indeed, if $\alpha$ is any curve that joins $1$ to $x^{-1}$ then the curve $\beta(t)=\alpha(t)^{-1}$ joins $1$ to $x$; if we derive we have $\dot{\beta}(t)\beta(t)^{-1}=-\alpha(t)^{-1}\dot{\alpha}(t)$ and then the right length of $\beta$ coincides with the left length of $\alpha$.
\end{proof}
If $\xi : [0,1] \rightarrow \mathcal{O}_{L_0}$ is a curve with $\xi(0)=L$ then a lifting of $\xi$ is a map $\phi : [0,1] \rightarrow {\rm Sp}_2(\mathcal{H})$ with $\phi(0)=1$ and $\phi(t)(L)=\xi(t)$, for all $t \in [0,1]$. The next lemma is an adaptation of Lemma 25 in \cite{Piccione1}. 

\begin{lem}\label{levantadas} Every smooth curve $\xi : [0,1] \rightarrow \mathcal{O}_{L_0}$ with $\xi(0)=L$ admits an isometric lifting, if we consider the right invariant metric in ${\rm Sp}_2(\mathcal{H})$.
\end{lem}
\begin{proof} For each $t\in [0,1]$, set $X(t)=-J\dot{\xi}(t)P_{\xi(t)} \in \mathfrak{sp}_2(\mathcal{H})$ and consider the solution of the ODE 
\begin{equation}\left\{ \begin{array}{lcl}\label{lev}
\dot{\phi}(t)=X(t)\phi(t) \\
& & \\
\phi(0)=1
\end{array}
\right.
\end{equation}
A simple computation using Proposition \ref{difpiL} shows that both $t\mapsto \phi(t)(L)$ and $\xi(t)$ are integral curves of the vector field $\nu(t)(L)=P_LJX(t)\vert_L \in T_L\mathcal{O}_{L_0}=\mathcal{B}_2(L)_s$ both starting at $L$, therefore the two curves coincide. Now, it is easy to see that the solution of the differential equation (\ref{lev}) is an isometric lifting of $\xi$. Indeed, if we take norms in the equation 
 we have, $$\Vert\dot{\phi}(t)\phi^{-1}(t)\Vert_2=\Vert-J\dot{\xi}(t)P_{\xi(t)}\Vert_2=\Vert\dot{\xi}(t)P_{\xi(t)}\Vert_2=\mathcal{A}(\xi(t),\dot{\xi}(t)).$$
\end{proof}

\begin{teo}\label{compda} The metric space $(\mathcal{O}_{L_0}, d_{\mathcal{A}})$ is complete.
\end{teo}
\begin{proof} Let $(L_n)$ be a $d_{\mathcal{A}}$-Cauchy sequence in $\mathcal{O}_{L_0}$ and fix $\varepsilon>0$. Then there exists $n_0$ such that $d_{\mathcal{A}}(L_n,L_m)\leq \varepsilon$ if $n,m \geq n_0$. For the fixed Lagrangian ${L_{n_0}}$, we have the map $$\pi=\pi_{L_{n_0}}: {\rm Sp}_2(\mathcal{H}) \rightarrow \mathcal{O}_{L_0} , \ \ \pi(g)=g(L_{n_0}).$$ If $n,m \geq n_0$ we can take a curve $\gamma_{n,m}\subset \mathcal{O}_{L_0}$ that joins $L_n$ to $L_m$ (for $t=0$ and $t=1$ respectively) such that $$L_{\mathcal{A}}(\gamma_{n,m})\leq d_{\mathcal{A}}(L_n,L_m)+\varepsilon.$$
Then by Lemma \ref{levantadas}, the curves $\gamma_{n_0,m}$ are lifted, via $\pi$, to curves $\phi_m$ of ${\rm Sp}_2(\mathcal{H})$ with $\phi_m(0)=1$ and $ L_{\mathcal{R}}(\phi_m)=L_{\mathcal{A}}(\gamma_{n_0,m})$. Denote by $g_m=\phi_m(1)\subset{\rm Sp}_2(\mathcal{H})$ the end point. Then 
$$\varepsilon+d_{\mathcal{A}}(L_{n_0},L_m)\geq L_{\mathcal{A}}(\gamma_{n_0,m})=L_{\mathcal{R}}(\phi_m)\geq d_{\mathcal{R}}(1,g_m).$$ 
For each $n,m \geq n_0$ we have, $$d_{\mathcal{R}}(g_n,g_m)\leq d_{\mathcal{R}}(1,g_m)+d_{\mathcal{R}}(1,g_n)\leq 2\varepsilon+ d_{\mathcal{A}}(L_{n_0},L_m)+d_{\mathcal{A}}(L_{n_0},L_n)\leq 4\varepsilon.$$ Thus the sequence $(g_m)\subset {\rm Sp}_2(\mathcal{H})$ is $d_{\mathcal{R}}$-Cauchy and then by Proposition \ref{dldr} we have that $(g_m^{-1})$ is $d_{\mathcal{L}}$-Cauchy. Using Lemma 7.1 of \cite{Manuel} we have that the sequence $(g_m^{-1})$ is a Cauchy sequence in $({\rm Sp}_2(\mathcal{H}),\Vert . \Vert_2)$ and then since this metric space is closed, there exists $x \in {\rm Sp}_2(\mathcal{H})$ such that $g_m^{-1} \stackrel{\Vert. \Vert_2}\longrightarrow x$. By continuity we have $\pi(g_m) \stackrel{\mathcal{O}_{L_0}}\longrightarrow \pi(x^{-1})$ and since $\phi_m$ is a lift of $\gamma_{n_0,m}$  we also have $\pi(g_m)=g_m(L_{n_0})=\phi_m(1)(L_{n_0})=\gamma_{n_0,m}(1)=L_m$, so $L_m \stackrel{\mathcal{O}_{L_0}}\longrightarrow \pi(x^{-1})$. Thus using Lemma \ref{convda} we have $d_{\mathcal{A}}(L_m,\pi(x^{-1}))\rightarrow 0$.
\end{proof}

The Riemannian connection given by the left invariant metric in the group ${\rm Sp}_2(\mathcal{H})$ was calculated in \cite{Manuel}. There was proved that the Riemannian connection in the group ${\rm Sp}_2(\mathcal{H})$ matches it the one of $GL_2(\mathcal{H})$; the group of invertible operators which are pertubations of the identity by a Hilbert-Schmidt operator. If $g_0 \in {\rm Sp}_2(\mathcal{H})$ and $g_0v_0 \in g_0.\mathfrak{sp}_2(\mathcal{H})$ are the initial position and the initial velocity then  $$\alpha(t)=g_0e^{tv_0^{*}}e^{t(v_0-v_0^{*})}\subset {\rm Sp}_2(\mathcal{H})$$ is a geodesic of the Riemannian connection. This fact can be used to find the geodesic of the Riemannian connection induced by the ambient metric $\mathcal{A}$.

\begin{teo} Let $\xi : [0,1] \rightarrow \mathcal{O}_{L_0}$ be a geodesic curve of the Riemannian connection induced by the ambient metric $\mathcal{A}$ with initial position $\xi(0)=L$ and initial velocity $\dot{\xi}(0)=w \in T_{\xi(0)} \mathcal{O}_{L_0}=\mathcal{B}_2(L)_s$. Then $$\xi(t)=e^{t(v^*-v)}e^{-tv^*}(L)$$ where $v\in \mathfrak{sp}_2(\mathcal{H})$ is a preimage of $-w$ by $d_1\pi_L$.
\end{teo}
\begin{proof} Since $\xi$ is a geodesic curve, by general considerations of Riemannian theory, it is locally minimizing. Using the Lemma \ref{levantadas} there exists an isometric lifting $\phi\subset {\rm Sp}_2(\mathcal{H})$ with initial condition $\phi(0)=1$. By the isometric property $\phi$ results locally minimizing with the right invariant metric and then $\phi^{-1}$ results locally minimizing with the left invariant metric. Hence the curve $\phi^{-1} \subset {\rm Sp}_2(\mathcal{H})$ is a geodesic and it is $\phi^{-1}(t)=e^{tv^{*}}e^{t(v-v^{*})}$ for some $v\in \mathfrak{sp}_2(\mathcal{H})$. Then it is clear that $\phi(t)=e^{t(v^*-v)}e^{-tv^*}$ and $\xi(t)=e^{t(v^*-v)}e^{-tv^*}(L)$. The only fact left to prove is that $v$ is a lift of $-w$. Indeed, since $\dot{\xi}(t)=d_{e^{t(v^*-v)}e^{-tv^*}}\pi_L\big((v^*-v)e^{t(v^*-v)}e^{-tv^*}-e^{t(v^*-v)}e^{-tv^*}v^*\big)$, then $w=\dot{\xi}(0)=d_1\pi_L(-v)=-d_1\pi_L(v)$.       
\end{proof}

\subsection{The quotient metric}
If $W \in \mathcal{O}_{L_0}$ and $v \in  T_W \mathcal{O}_{L_0}$, we put $$\mathcal{Q }(W,v)=\inf \lbrace \Vert z\Vert_2: z\in  \mathfrak{sp}_2(\mathcal{H}), \   d_1\pi_W(z)=v\rbrace.$$ 
This metric will be called the quotient metric of $\mathcal{O}_{L_0}$, because it is the quotient metric in the Banach space $$T_W \mathcal{O}_{L_0}    \simeq \mathfrak{sp}_2(\mathcal{H})/\mathfrak{sp}_2(\mathcal{H})_W.$$ Indeed, since $\mathfrak{sp}_2(\mathcal{H})_W=\ker d_1\pi_W$, if $z\in  \mathfrak{sp}_2(\mathcal{H})$ with   $d_1\pi_W(z)=v$ then $$\mathcal{Q }(W,v)=\inf \lbrace \Vert z-y\Vert_2 : y \in \mathfrak{sp}_2(\mathcal{H})_W\rbrace.$$ If $Q_L$ denotes the orthogonal projection onto $\mathfrak{sp}_2(\mathcal{H})_W$ then each $z\in \mathfrak{sp}_2(\mathcal{H})$ can be uniquely decomposed as $$z=z-Q_L(z)+Q_L(z)=z_0+Q_L(z)$$ hence $$\Vert z-y\Vert_2^2=\Vert z_0+Q_L(z)-y\Vert_2^2=\Vert z_0\Vert_2^2+\Vert Q_L(z)-y\Vert_2^2\geq \Vert z_0\Vert_2^2$$ for any $y \in \mathfrak{sp}_2(\mathcal{H})_W$ which shows that 
\begin{equation}\label{almin} \mathcal{Q }(W,v)=\Vert z_0\Vert_2 \end{equation}
where $z_0$ is the unique vector in ${\mathfrak{sp}_2(\mathcal{H})}_W^{\perp}$ such that $d_1\pi_W(z_0)=v$.

We denote the length for a piecewise smooth curve in $\mathcal{O}_{L_0}$, measured with the quotient norm introduced above, $$L_{Q}(\gamma)=\int_0^1 \mathcal{Q }(\gamma(t),\dot{\gamma}(t)) dt$$ 
and by $d_{\mathcal{Q}}$ the geodesic distance in $\mathcal{O}_{L_0}$  
$$d_{\mathcal{Q}}(S,T)=\inf \lbrace L_{\mathcal{Q}}(\gamma) : \gamma \ \mbox{joins} \ S  \ \mbox{and} \  T  \ \mbox{in} \ \mathcal{O}_{L_0}\rbrace.$$

\begin{prop}\label{desdadq} $d_{\mathcal{A}}(S,T)\leq d_{\mathcal{Q}}(S,T)$ for all $S,T\in \mathcal{O}_{L_0}$. 
\end{prop}
\begin{proof} The proof is a straighforward computation using the definition of the metrics; indeed let $\gamma$ be any curve that joins $S$ with $T$, since $\mathcal{Q}(\gamma(t),\dot{\gamma}(t))=\Vert \alpha(t)\Vert_2$ where $d_1\pi_\gamma(\alpha)=\dot{\gamma}$ then, $$\mathcal{A}(\gamma,\dot{\gamma})=\Vert \dot{\gamma}P_{\gamma}\Vert_2=\Vert d_1\pi_\gamma(\alpha)P_{\gamma}\Vert_2=\Vert P_{\gamma}J\alpha\vert_{\gamma}P_{\gamma}\Vert_2\leq \Vert \alpha\Vert_2=\mathcal{Q}(\gamma,\dot{\gamma}).$$
\end{proof}
 
\begin{lem} Let $(L_n)\subset \mathcal{O}_{L_0}$ such that $L_n \stackrel{\mathcal{O}_{L_0}}\rightarrow L $; then $L_n \stackrel{d_{\mathcal{Q}}}\longrightarrow L.$
\end{lem} 
\begin{proof} The proof is similar to that of Proposition \ref{convda}. Since the map $\pi_L$ has local continuous section, let $n_0$ such that $L_n \in U \subset\mathcal{O}_{L_0} \ \forall n \geq n_0$ ( $U$ a neighboord of $L$) and such that $\sigma_L : U \rightarrow {\rm Sp}_2(\mathcal{H})$ is a section for $\pi_L$. By continuity we have  $\sigma_L(L_n) \stackrel{\Vert . \Vert_2}\longrightarrow \sigma_L(L)=1$ if $n\geq n_0$. Since $\sigma_L(L_n)$ is close to $1$, there is $z_n \in \mathfrak{sp}_2(\mathcal{H})$ such that $\sigma_L(L_n)=e^{z_n}$ and since $\Vert e^{z_n} -1\Vert_2=\Vert \sigma_L(L_n) -1 \Vert_2\rightarrow 0$ we also have $\Vert z_n \Vert_2 \rightarrow 0$. Let $\gamma_n(t)=e^{tz_n}(L)\subset \mathcal{O}_{L_0}$ a curve that joins $L$ and $L_n$. By the formula (\ref{almin}), $$Q(\gamma_n,\dot{\gamma}_n)=\Vert x_n(t)\Vert_2$$ where $x_n(t)$ is the unique vector in ${\mathfrak{sp}_2(\mathcal{H})}_{\gamma_n(t)}^{\perp}$ such that $d_1\pi_{\gamma_n(t)}(x_n(t))=\dot{\gamma_n}(t)$. Then using the chain rule and  Proposition \ref{difpiL} the above equality shows 
\begin{equation}\label{znxn} P_{e^{tz_n}(L)}Jx_n(t)\vert_{e^{tz_n}(L)}=P_{e^{tz_n}(L)}Jz_n\vert_{e^{tz_n}(L)}.
\end{equation} 
This means that the compression of $Jx_n(t)$ to the 1-1 position block in the decomposition $e^{tz_n}(L)\oplus e^{tz_n}(L)^{\perp}$ is equal to the compression of $Jz_n$ for all $n,t$. Since $x_n(t)$ belongs in ${\mathfrak{sp}_2(\mathcal{H})}_{e^{tz_n}(L)}^{\perp}$, we can write it  $$x_n(t)=x_n(t)P_{e^{tz_n}(L)}-P_{e^{tz_n}(L)}x_n(t)P_{e^{tz_n}(L)}=(1-P_{e^{tz_n}(L)})x_n(t)P_{e^{tz_n}(L)}$$ and then since $P_{e^{tz_n}(L)}$ is a Lagrangian projector we have $$Jx_n(t)=(J-JP_{e^{tz_n}(L)})x_n(t)P_{e^{tz_n}(L)}=P_{e^{tz_n}(L)}Jx_n(t)P_{e^{tz_n}(L)}.$$ Thus using the equality (\ref{znxn}) $$ Q(\gamma_n,\dot{\gamma}_n)=\Vert x_n(t)\Vert_2=\Vert Jx_n(t)\Vert_2=\Vert P_{e^{tz_n}(L)}Jx_n(t)P_{e^{tz_n}(L)}\Vert_2$$ $$=\Vert P_{e^{tz_n}(L)}Jz_n\vert_{e^{tz_n}(L)}\Vert_2 \leq \Vert z_n\Vert_2$$ for all $n,t$. Then is clear that $d_{\mathcal{Q}}(L_n,L)\leq L_{\mathcal{Q}}(\gamma_n)\rightarrow 0$.
\end{proof}
Now we are in a position to obtain our main result.
\begin{teo} The metric space $(\mathcal{O}_{L_0},d_{\mathcal{Q}})$ is complete.
\end{teo}
\begin{proof} Let $(L_n)$ be a $d_{\mathcal{Q}}$-cauchy sequence, then by Proposition \ref{desdadq} $(L_n)$ is a $d_{\mathcal{A}}$-cauchy sequence. If we repeat the procedure that we did in Theorem \ref{compda} we have that $L_n \stackrel{\mathcal{O}_{L_0}}\longrightarrow \pi(x^{-1})$ for some $x\in {\rm Sp}_2(\mathcal{H})$ and then using the above lemma we have $d_{\mathcal{Q}}(L_n,\pi(x^{-1}))\rightarrow 0$. 

\end{proof}


\paragraph{Acknowledgements}
I want to thank Prof. E. Andruchow and Prof. G. Larotonda for their suggestions and support.

\bigskip
{\footnotesize Manuel L\'opez Galv\'an.\\
Instituto de Ciencias, Universidad Nacional de General Sarmiento.\\
JM Guti\'errez 1150 (1613) Los Polvorines. Buenos Aires, Argentina.\\
e-mail: mlopezgalvan@hotmail.com}


\begin{thebibliography}{XXXX}

\bibitem{Andruchow1} E. Andruchow and G. Larotonda. Lagrangian Grassmannian in infinite dimension. J. Geom. Phys. 59 (2009), no. 3, 306-320.

\bibitem{Andruchow2} E. Andruchow, G. Larotonda, L. Recht. Finsler geometry and actions of the p-Schatten unitary groups. Trans. Amer. Math. Soc. 362 (2010), 319-344.

\bibitem{Arnold} V.I. Arnold. On a characteristic class entering into conditions of quantization, Funkcional. Anal. i Prilo$\check{z}$en. 1 (1967) 1-14 (in Russian).

\bibitem{Atkin} C.J. Atkin. The Hopf-Rinow theorem is false in infinite dimensions. Bull. London Math. Soc. 7 (1975), 261-266.

\bibitem{Beltita} D. Belti\c{t}$\check{a}$. Smooth homogeneous structures in operator theory. Chapman and Hall/CRC. Monographs and Surveys in Pure and Applied Mathematics, 137. Chapman and Hall/CRC, Boca Raton, FL, 2006.

\bibitem{Piccione1} L. Biliotti, R. Exel, P. Piccione and D. V. Tausk. On The Singularities of the exponential map in
infinite dimensional Riemannian Manifolds. Math. Ann. 336 (2) (2006) 247-267.

\bibitem{Furutani} K. Furutani. Fredholm-Lagrangian-Grassmannian and the Maslov index. J. Geom. Phys. 51 (3) (2004) 269-331. 

\bibitem{Harpe} Pierre de la Harpe. Classical Banach-Lie Algebras and Banach-Lie Groups of Operators in Hilbert Space. Springer-Verlag.
Berlin. Heidelberg. NewYork 1972.

\bibitem{Lang} S. Lang. Differentiable and Riemannian manifolds. Third edition. Graduate Texts in Mathematics, 160. Springer-Verlag, New York, 1995.

\bibitem{Manuel} M. L\'opez Galv\'an. Riemannian metrics on an infinite dimensional Symplectic group. Journal of Mathematical Analysis and Applications (2015), in press. doi:10.1016/j.jmaa.2015.03.051.

\bibitem{McAlpin} J. McAlpin. Infinite dimensional manifolds and Morse theory. Thesis, Columbia University, 1965.

\bibitem{Piccione2} P. Piccione and D. Victor Tausk. A Student's Guide to Symplectic Spaces, Grassmannians and Maslov Index. Publica\c{c}\~oes Matem\'aticas do IMPA. Rio de Janeiro, Instituto Nacional de Matem\'atica Pura e Aplicada (IMPA). xiv, 301 p. (2008).
\end{thebibliography}
\end{document}